\documentclass[12pt]{amsart}
\usepackage[utf8]{inputenc}
\usepackage[english]{babel}

\usepackage{amsmath,amssymb,amsthm,amsfonts}
\usepackage{hyperref}
\usepackage{bbm}
\usepackage{dsfont}

\newcommand{\overbar}[1]{\,\overline{\!{#1}}}

\usepackage[margin=1in]{geometry}

\usepackage{graphicx}
\graphicspath{ {./images/} }
\usepackage{tikz}
\definecolor{myblue}{RGB}{39, 106, 179}
\definecolor{myred}{RGB}{190, 1, 25}
\definecolor{myyellow}{RGB}{255, 173, 1}

\newtheorem{thm}{Theorem}[section]
\newtheorem{lem}[thm]{Lemma}
\newtheorem{prop}[thm]{Proposition}
\newtheorem{defn}[thm]{Definition}

\theoremstyle{remark}

\newcommand{\supp}{\mathrm{supp}}
\newcommand{\mcE}{\mathcal{E}}
\newcommand{\mcA}{\mathcal{A}}

\title{Formation of multiple flocks in a simple nonlocal aggregation model}
\author{Carrie Clark}\thanks{This work was supported in part by NSERC Discovery Grant \#311685} 

\begin{document}

\maketitle

\begin{abstract}
We consider a family of interaction energies given by kernels having a ``well-barrier" shape, and investigate how these kernels drive the formation of multiple flocks within a larger population. We show that the energy is minimized by a sequence of indicators of finitely many balls whose supports become infinitely far apart from one another. The dichotomy case of the concentration compactness principle is a key ingredient in our proof. 

\

\end{abstract}

\section{Introduction}

The study of aggregation phenomena in biology (see for example \cite{Bernoff_Topaz_2013}\cite{Topaz_Bertozzi_Lewis_2006}) and physics (see for example \cite{Wales_2010}\cite{Holm_Putkaradze_2005}) has produced an interesting class
of geometric shape optimization problems. Large scale collective behaviour, such as collaboration and formation into flocks, is shaped at least in part by social forces between individuals. 

In this paper we will consider isotropic interaction kernels, that is we consider the case where the attractive and repulsive forces between individuals are determined entirely by the distances between individuals. Even in this simple case there are a wide range of complex phenomena.  For instance, Carrillo, Figalli and Patacchini \cite{Carrillo_Figalli_Patacchini} have shown that weak repulsion at short distances implies that global minimizers of the interaction energy have finite support.  Lim and McCann \cite{Lim_McCann} have shown that for certain power law kernels with mild repulsion and strong attraction, the interaction energy is uniquely minimized by measures that are uniformly distributed on the vertices of a unit simplex. For certain power law kernels, Lopes \cite{Lopes} has shown that minimizers are unique and radially symmetric. For further examples, see  \cite{Balague-et-al-dimensionality}, \cite{Carrillo_Delgadino_Mellet_2016}, and \cite{frank2019proof}. 

We focus on the particular phenomenon of formation of multiple flocks within a large population. To this end, we consider interactions that are attractive at short distances, followed by a repulsive ``barrier" at mid distances, and neutral at long distances. See Figure~\ref{fig:K sketch} for a sketch of such an interaction kernel. The short range attraction motivates individuals to flock together, while the mid range repulsion prevents the entire population from forming into a single flock. In biological applications, the aggregation problem typically involves kernels that are repulsive at short distances, to prevent collisions between individuals, for example see \cite{Mogilner_Bent_Spiros_Edelstein-Keshet_2003}.  In this paper, we use a density constraint to prevent the population from collapsing to a point, and to force the population to spread out as its size increases. So, while there is no short-range repulsion built into the kernel, the density constraint provides a hard-core repulsion.

Many existing models for which minimizers are known to exist have kernels that grow at infinity, for example Theorem 2.1 and Remark 2.2 in \cite{Choksi_Fetecau_Topaloglu_2015} show existence for a general family of such kernels. We dispose of this assumption of growth at infinity, and this lack of coercivity presents an additional technical challenge in proving existence of minimizers.

\begin{figure}
    \centering
    \begin{tikzpicture}[scale=0.75]
    \draw[thick,->] (0,0) -- (7.18,0)node[right] {$r$};
    \draw[thick,->] (0,0) -- (0,4);
    \draw[thick] (0,0) -- (0,-2);
    
    \draw (0,-1.5) .. controls (0.65,-1.5) and (0.9,-0.75) .. (1,0);
    \draw (1,0) .. controls (1.5,3) and (2,3.55) .. (3,3.65);
    \draw (3,3.65) .. controls (4,3.55) .. (5,1.2);
    \draw[->] (5,1.2) .. controls (5.5,0.1) and (6,0.1) .. (7,0.1);
    \draw (5,3.5)node{$K(r)$};
    
    \draw (1,0.15) -- (1,-0.15)node[anchor= north west] {$w$}; 
    \draw (1.75,0.15) -- (1.75,-0.15)node[anchor= north west] {$a$}; 
    \draw (4.25,0.15) -- (4.25,-0.15)node[below] {$a+W$}; 
    
    \draw (-0.15,-1.5)node[left] {$-d$} -- (0.15,-1.5);
    \draw (-0.15,2.85)node[left] {$h$} -- (0.15,2.85);
    \end{tikzpicture}
    \caption{Sketch of a ``well-barrier" type kernel. The parameters $d$ and $w$ represent the depth and width of the attractive well, and the parameters $h$ and $W$ represent the height and width of the repulsive barrier.}
    \label{fig:K sketch}
\end{figure}
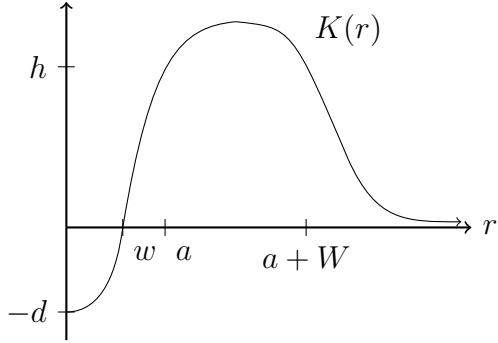

We show that, under certain conditions on the relative sizes of the attractive well and repulsive barrier, there are minimizing sequences of densities that are indicator functions of balls, whose centres get infinitely far apart from one another. To show this, we use the concentration compactness principle to show that minimizers exist in the case where distant individuals have no mutual interaction. The dichotomy case of the concentration compactness principle is our main tool for extracting a minimizer even though an arbitrary minimizing sequence is expected to split into many pieces whose supports may be far away from each other.

\

\subsection{Statement of Main Results}

Consider the problem of minimizing the energy
\begin{equation*}
    \mcE[\rho]=\int_{\mathbb{R}^N}\int_{\mathbb{R}^N} K(|x-y|)\rho(x)\rho(y)\, dx \, dy
\end{equation*}
given by an interaction kernel $K$ over the set
\begin{equation*}
    \mcA_m=\{ \rho\in L^1(\mathbb{R}^N) \; : \; 0\leq \rho \leq 1, \; ||\rho||_{L^1(\mathbb{R}^N)}=m \}
\end{equation*} of densities having total mass m. 
This nonlocal interaction energy arises in the study of aggregation in biology, as outlined in \cite{Bernoff_Topaz_2013}.

We will be considering a certain family of kernels. The key features of these kernels are that they are attractive at short distances, followed by a repulsive ``barrier," followed by decay at infinity.  More precisely, we will only consider kernels $K: [0,\infty) \to \mathbb{R}\cup\{+\infty\}$ which have a ``well-barrier" shape, as given by the following conditions:

\begin{enumerate}
    \item[\textbf{(K1)}] \textbf{Well.} $K(0)<0,$ and $K$ is non-decreasing on  some interval $[0,a],$
    \item[\textbf{(K2)}] \textbf{Barrier.} $K(r)\geq h$ on $[a,a+W],$ for some width $W>0,$ and
    \item[\textbf{(K3)}] \textbf{Decay.} $K(r)\geq0$ on $(a+W,\infty),$ and $\lim_{r\to \infty}K(r)=0$
\end{enumerate}

It turns out, for large mass, minimizers do not exist, but we can construct minimizing sequences that consists of indicators of finitely many disjoint balls that become increasingly far apart. 

To get a separation into multiple ``droplets'' or ``flocks'' for general well-barrier type kernels, we will make the following additional assumptions on the relative sizes of the attractive well and repulsive barrier. Let $d=-K(0),$ which represents the depth of the well, and let $w=\inf\{r>0:K(r)>0 \}$ be the width of the well.

\begin{enumerate}
    \item[\textbf{(K4)}]\textbf{Barrier height.}  $d<h$, and
    \item[\textbf{(K5)}]\textbf{Barrier width.} $W\leq a+3w$.
\end{enumerate}

Due to the kernel's decay at infinity, minimizers do not generally exist, but we can still consider the minimal energy:

\begin{equation*}
    E(m)=\inf_{\rho\in \mcA_m} \mcE[\rho].
\end{equation*}

We prove that this minimal energy is attained by a sequence of densities which are indicators of many balls which become infinitely far apart from one another.

\begin{thm}\label{thm:inf is many balls}
Let $K$ be a bounded kernel which satisfies $\mathbf{(K1)-(K5)}$.
Further, assume $K$ is strictly increasing on $[0,a]$. Then for any $m>0,$ there exists some $k\in \mathbb{N}$ and radii $r_1, \dots, r_k >0$ such that
\begin{equation}
    E(m)= \sum_{i=1}^k\mcE[\mathds{1}_{B(0,r_i)}].
\end{equation}

\end{thm}

We prove Theorem \ref{thm:inf is many balls} by first considering the case where the kernel $K$ is compactly supported. In this case, once the distance between droplets is large enough there is no longer a benefit to moving further apart, so minimizers do exist. The main tool for proving this existence is Lions' concentration compactness principle. We start with a minimizing sequence, and concentration compactness tells us that there is a subsequence that is vanishing, tight up to translation, or dichotomous. A common strategy when using concentration compactness is to rule out the possibility of a vanishing or dichotomous minimizing sequence. The resulting tight minimizing sequence is then used to show that a minimizer exists. In our case, we cannot rule out dichotomy, instead we use it to decompose the minimizing sequence into finitely many pieces whose supports become very far away from each other. Further, each of these pieces are tight up to translation. Then, since the kernel is zero outside of some large radius, these pieces do not actually need to spread out far away from each other in order to reach the minimal energy, allowing us to construct a minimizer.

Then to prove Theorem \ref{thm:inf is many balls}, we obtain a minimizing sequence by truncating the kernel and then sending the pieces of the minimizer for the truncated problem infinitely far away from each other. In \ref{section:gen min} we analyze this using the notion of generalized minimizers.

To better understand the underlying geometry, one may consider a toy version of this problem where the barrier is infinitely high, the attractive well is a constant function, and the kernel is zero elsewhere. More precisely, the kernel K is $-1$ on $[0,1),$ followed by $+\infty$ on an interval $[1,1+W],$ and zero elsewhere.  In this simplification, the only parameter is the width of the repulsive barrier relative to the attractive well. If the barrier is wide enough, then any density with finite energy can be easily decomposed into finitely many well separated pieces whose supports have diameter at most $1.$ However, the geometric problem becomes more complicated when the repulsive barrier is narrower. In \cite{carrie-thesis}, we showed that no matter the width of the forbidden range of distances, the minimal energy in the one dimensional problem is always attained by a well separated union of intervals of length $1$, plus possibly a piece of smaller mass which is supported in an interval of length one. However, in the $W=0$ case, there are examples of configurations which attain the minimal energy and which cannot be decomposed into pieces that do not interact with one another.

In \ref{subsec:lingrowth} we show that the minimal energy grows linearly in $m,$ which suggests that for large mass, the droplets in a minimizer are all relatively close in size. In Section \ref{sec:sizeofdrop} we investigate the sizes of the droplets in the specific example where the kernel is a power law near $0.$ We show that generalized minimizers consist of many balls of the same size, and at most one other smaller ball.

\section{Preliminaries on shape of minimizers}

Throughout this section we consider kernels $K$ which satisfy the well-barrier conditions \textbf{(K1)-(K3)}. Additionally, we will need the following notation to denote the interaction energy between two densities $\rho$ and $\eta$
\begin{equation*}
    \mcE[\rho,\eta]=\int_{\mathbb{R}^N}\int_{\mathbb{R}^N}K(|x-y|)\rho(x)\eta(y) \, dx \, dy.
\end{equation*}
Using this notation we will write $\mcE[\rho+\eta]=\mcE[\rho]+2\mcE[\rho,\eta]+\mcE[\eta].$

\subsection{Balls are minimal for small mass}
First we show that indicators of balls are minimal for small mass. 
\begin{lem}\label{lem:smallmassballmin}
There is a mass $m_0>0$, such that for all $m\leq m_0$
\begin{equation}
    \inf_{\rho \in \mcA_m}\mcE[\rho] = \mcE[\mathds{1}_{B(0,r)}]
\end{equation}
where $B(0,r)$ is the ball of measure $m$.
\end{lem}

\begin{proof}
Fix $m>0$ and let $\rho \in \mcA_m$ and let $r$ be the radius of a ball with measure $m.$
First note that 
\begin{equation*}
    K(r)\leq \mathds{1}_{[0,w]}(r) K(r).
\end{equation*}
The right hand side, $ \mathds{1}_{[0,w]}(r) K(r)$ is non-decreasing, and so we may apply the Riesz rearrangement inequality (see \cite{Lieb_Loss_Society_1997} Ch.3), which says that
\begin{equation*}
    \int\limits_{\mathbb{R}^N}\int\limits_{\mathbb{R}^N}K(|x-y|)\mathds{1}_{|x-y|\leq w} \rho(x) \rho (y)\, dx \, dy
     \geq     \int\limits_{\mathbb{R}^N}\int\limits_{\mathbb{R}^N}K(|x-y|)\mathds{1}_{|x-y|\leq w} \rho^*(x) \rho^* (y)\, dx \, dy
\end{equation*}
where $\rho^*$ is the symmetrically decreasing rearrangement of $\rho$. By the bathtub principle (see \cite{Lieb_Loss_Society_1997} Theorem 1.14) we have 
\begin{equation*}
    \int\limits_{\mathbb{R}^N}\int\limits_{\mathbb{R}^N}K(|x-y|)\mathds{1}_{|x-y|\leq w} \rho^*(x) \rho^* (y)\, dx \, dy \geq 
    \int\limits_{\mathbb{R}^N}\int\limits_{\mathbb{R}^N}K(|x-y|)\mathds{1}_{|x-y|\leq w} \mathds{1}_{B(0,r)}(x) \mathds{1}_{B(0,r)} (y)\, dx \, dy.
\end{equation*}
Finally, if $r\leq w,$ then 
\begin{equation*}
    \int\limits_{\mathbb{R}^N}\int\limits_{\mathbb{R}^N}K(|x-y|)\mathds{1}_{|x-y|\leq w} \mathds{1}_{B(0,r)}(x) \mathds{1}_{B(0,r)} (y)\, dx \, dy=\mcE[\mathds{1}_{B(0,r)}].
\end{equation*}
So, if $m$ is small enough so that $r\leq w,$ we have that $\mcE[\rho] \geq \mcE[\mathds{1}_{B(0,r)}]$ for all $\rho \in \mcA_m$.
\end{proof}

\subsection{Minimizers have compact support}

To prove that minimizing densities have compact support, we will make use of the following Euler-Lagrange conditions. These minimality conditions will also play a key role in showing that points in the support of a minimizing density cannot be within certain range of distances from one another.

\begin{lem}\label{nec-cond}
Let $\rho \in \mcA_m$ be a local minimizer of the energy $\mcE,$ in the $L^1$ topology. Then there is a $\lambda < 0$ such that for almost every $x$ we have
\begin{equation*}
    K*\rho(x) 
    \left\{
    \begin{array}{lr}
        \geq \lambda, & \; \text{ if } \rho(x)=0,\\
        = \lambda, & \; \text{ if } 0<\rho(x)<1,\\
        \leq \lambda, & \; \text{ if } \rho(x)=1.
    \end{array} 
    \right.
\end{equation*}
\end{lem}

\begin{proof}
The proof follows the same method as in \cite{Burchard_Choksi_Topaloglu_2018}, with some modifications to account for the kernel $K$ taking the value $+\infty$.
Let $S_0=\{\rho=0\}$, and $S_1=\{\rho=1\}$. We want to construct a perturbation of $\rho,$ by adding to $\rho$ on $S_0,$ and subtracting on $S_1.$ Fix any nonnegative, bounded, compactly supported $\varphi$ and $\psi \in L^1(\mathbb{R}^d)$ with $||\varphi||_{L^1(\mathbb{R}^d)}=||\psi||_{L^1(\mathbb{R}^d)=1}$, and $\varphi=0$ a.e. in $S_1$ and $\psi=0$ a.e. in $S_0$. Note that $\mcE[\psi,\rho]<\infty,$ since $\psi=0$ almost everywhere in $S_0,$ and $\mcE[\rho]<\infty.$ For now assume that $\mcE[\varphi,\rho]<\infty$. Let $\epsilon>0$ and define
\begin{equation*}
    \varphi_{\epsilon}(x)=\frac{1}{||\varphi \mathds{1}_{\rho<1-\epsilon}||_{L^1(\mathbb{R}^N)}}\varphi(x)\mathds{1}_{\rho<1-\epsilon}(x),
\end{equation*}
\begin{equation*}
    \psi_{\epsilon}(x)=\frac{1}{||\psi \mathds{1}_{\rho>\epsilon}||_{L^1(\mathbb{R}^N)}}\psi(x)\mathds{1}_{\rho>\epsilon}(x).
\end{equation*}
Then $\eta_t=\rho+t(\varphi_{\epsilon}-\psi_{\epsilon})\in \mcA_m$. For sufficiently small $t>0$ we have  $\mcE[\eta_t]\geq \mcE[\rho],$ since $\rho$ is a local minimizer. Then,
\begin{align*}
    0 &\leq \lim_{t \to 0^+}\frac{\mcE[\rho+t(\varphi_{\epsilon}-\psi_{\epsilon})]-\mcE[\rho]}{t} \\
    & =\lim_{t \to 0^+} \frac{\mcE[\rho]+2t\mcE[\rho,\varphi_{\epsilon}-\psi_{\epsilon}]+t^2\mcE[\varphi_{\epsilon}-\psi_{\epsilon}]-\mcE[\rho]}{t} \\
    &=2\int_{\mathbb{R}^N}K*\rho(x)(\varphi_{\epsilon}(x)-\psi_{\epsilon}(x))\,dx.
\end{align*}
Then taking $\epsilon \to 0,$ by dominated convergence we have
\begin{equation}\label{necinq}
    \int_{\mathbb{R}^N}K*\rho(x)\varphi(x)\,dx \geq \int_{\mathbb{R}^N}K*\rho(x)\psi(x)\,dx.
\end{equation}
The inequality (\ref{necinq}) also holds for any nonnegative $\varphi,\psi\in L^1(\mathbb{R}^N)$ that satisfy $||\varphi||_{L^1(\mathbb{R}^N)}=||\psi||_{L^1(\mathbb{R}^N)}=1,$ $\psi=0$ a.e. in $S_0$, $\varphi=0$ a.e in $S_1$, and $\mcE[\varphi,\rho]<\infty,$ by density of bounded compactly supported functions in $L^1(\mathbb{R}^N)$. Finally, we can relax the assumption that $\mcE[\varphi,\rho]<\infty,$ since (\ref{necinq}) is trivial in this case. Now let
\begin{equation}
    \lambda=\sup\left\{ \int_{\mathbb{R}^N}K*\rho(x)\psi(x)\,dx \; : \; ||\psi||_{L^1(\mathbb{R}^N)}=1, \; \psi \geq 0; \; \psi=0 \; \text{ a.e. in} \; S_0 \right\}.
\end{equation}
Then, by (\ref{necinq})
\begin{equation}
    \lambda \leq \inf\left\{ \int_{\mathbb{R}^N}K*\rho(x)\varphi(x)\,dx \; : \; ||\varphi||_{L^1(\mathbb{R}^N)}=1, \; \varphi \geq 0; \; \varphi=0 \; \text{ a.e. in} \; S_1 \right\}.
\end{equation}
These two equations tell us that$K*\rho(x) \leq \lambda$ for a.e $x$ such that $\rho(x)>0$ and $K*p(x)\geq \lambda$ for a.e. $x$ such that $\rho(x)<1$,  respectively.

To see that $\lambda< 0,$ suppose for contradiction that $K*\rho(x)\geq0$ on a set of positive measure $A\subset \{\rho >0 \}.$ Without loss of generality, assume $A$ has small enough diameter so that $\mcE[\rho|_A]< 0$. Then, 
\begin{equation*}
    \mcE[\rho-\rho|_A]=\mcE[\rho]-2\mcE[\rho,\rho|_A]+\mcE[\rho_A] < \mcE[\rho].
\end{equation*}
This says that we can improve $\rho$ by removing a small piece of it. We can construct a competitor by adding a small piece at infinity. To make the perturbation from $\rho$ small, we can pick $A$ to be as small as we like. 
\end{proof}

Again, following the general argument from \cite{Burchard_Choksi_Topaloglu_2018}, we can now prove that local minimizers are compactly supported.

\begin{lem}[$L^1$ local minimizers have compact support] \label{lem:mins have compact support}
If $\rho \in \mcA_m$ is a local minimizer of the energy $\mcE$ in the $L^1$ topology, then $\rho$ is compactly supported.
\end{lem}

\begin{proof}
Let $\rho \in \mcA_m$ be a local minimizer of the energy $\mcE,$ in the $L^1$ topology. Then by Lemma $\ref{nec-cond},$ there is a $\lambda<0$ such that, up to a set of measure zero, $\{\rho>0\}\subseteq \{K*\rho\leq \lambda \}$
For any $\epsilon>0$, let $R>0$ be large enough so that 
\begin{equation*}
    \int_{|y|>R} \rho(y)\, dy <\epsilon.
\end{equation*}
Let $r_0>0$ be such that $K(r)\geq 0$ for all $r>r_0$. Then, noting that $K(r)$ takes its minimum value at $r=0,$ we have
\begin{align*}
    K*\rho (x) & \geq \int_{|x-y|<r_0}K(0)\rho(y) \, dy \\
    & > K(0) \epsilon
\end{align*}
for all $|x|>R+r_0.$
Thus
\begin{equation*}
    \liminf_{|x| \to \infty} K*\rho(x) \geq 0.
\end{equation*}
But, since $\lambda<0,$ this means that the set $\{K*\rho(x) \leq \lambda \}$ is bounded. 
\end{proof}

\subsection{Subadditivity}
For $m>0,$ let 
\begin{equation*}
    E(m)=\inf_{\rho \in \mcA_m} \mcE[\rho].
\end{equation*}
We will show that this is a subadditive function of the mass $m.$
\begin{lem}[Subadditivity of the minimal energy]\label{lem:subadd}
Let $K:[0,\infty) \to \mathbb{R}$ be a bounded kernel satisfying the ``well-barrier" conditions $\mathbf{(K1)-(K3)}$. Let $m,n>0.$ Then,
\begin{equation*}
E(m+n) \leq E(m)+E(n)\end{equation*}
\end{lem}

\begin{proof}
Let $\{\rho_k^m \}\subset \mcA_m$ and $\{\rho_k^n \}\subset \mcA_n$ be minimizing sequences for $\inf_{\rho \in \mcA_m} \mcE[\rho]$ and $\inf_{\rho \in \mcA_n} \mcE[\rho]$, respectively. From these two sequences we will construct a sequence $\{ \rho_k\}\subset \mcA_{m+n},$ for which 
\begin{equation*}
    \lim_{k\to \infty} \mcE[\rho_k] = \lim_{k\to \infty} \mcE[\rho_k^m] + \lim_{k\to \infty} \mcE[\rho_k^n].
\end{equation*}
Our approach will be essentially to add a translated copy of $\rho_k^n$ to $\rho_k^m$. The translations will be chosen so that the translated $\rho_k^n$ has small interaction with $\rho_k^m.$ Since these densities may not be compactly supported, we will consider their restrictions to balls of suitably large radius, and add a small corrective term to ensure the resulting density is in $\mcA_{m+n}$.
For each $k\in \mathbb{N},$ there is a radius $R_{k}>0$ for which
\begin{equation}\label{bigradius}
    \int_{B(0,R_{k})} \rho_k^m \geq m-\frac{1}{k} \; \; \; \text{ and } \; \; \; \int_{B(0,R_{k})} \rho_k^n \geq n-\frac{1}{k}.
\end{equation}
Next, define
\begin{equation*}
    \overline{\rho}_k^m:= \rho_k^m|_{B(0,R_k)} \; \; \; \text{ and } \; \; \; \overline{\rho}_k^n:= \rho_k^n|_{B(0,R_k)}.
\end{equation*}
Finally, set
\begin{equation*}
    \rho_k(x) := \overline{\rho}_k^m  + \overline{\rho}_k^n (\cdot+x_k) + \mathds{1}_B(x_k,r_k)
\end{equation*}
where $r_k\geq 0$ is chosen so that $\rho_k \in \mcA_{m+n}$, and $x_k=(k+2R_k)e_1$. This choice for $x_k$ means that the supports of these three pieces become arbitrarily far apart as $k\to \infty.$ 
Now, we may expand
\begin{align}
    \mcE[\rho_k]= & \mcE[\overline{\rho}_k^m]+\mcE[\overline{\rho}_k^n] + \mcE[\mathds{1}_{B(x_k,r_k)}] \nonumber \\
     & +2\mcE[\overline{\rho}_k^m,\overline{\rho}_k^n(\cdot + x_k)] +2\mcE[\overline{\rho}_k^m,\mathds{1}_{B(x_k,r_k)}] \label{eqn:big}\\ 
     & + 2\mcE[\overline{\rho}_k^n(\cdot + x_k),\mathds{1}_{B(x_k,r_k)}]. \nonumber 
\end{align}
First, we will check that the final three terms in (\ref{eqn:big}) approach $0$ as $k\to \infty$. By our choice of $x_k$, $\text{dist}(B(0,R_k),B(-x_k,R_k))\geq k$, hence
\begin{align}
    |\mcE[\overline{\rho}_k^m,\overline{\rho}_k^n(\cdot + x_k)]| = & 
    \left|\int_{B(0,R_k)}\int_{B(-x_k,R_k)} K(|x-y|) \rho_k^n(x+x_k)\rho_k^m(y) \; dx \; dy \right| \\
    = & 
    \left|\int_{B(0,R_k)}\int_{B(-x_k,R_k)} K(|x-y|) \mathds{1}_{|x-y|\geq k} \rho_k^n(x+x_k)\rho_k^m(y) \; dx \; dy \right| \\
    \leq & mn\sup_{r\geq k}K(r).
\end{align}
 To conclude, use the fact that $K(r)$ tends to $0$ as $r\to \infty.$ Similar estimates can be done for the other two pair interaction terms.

Next, we need to check that the correction term $\mathds{1}_{B(x_k,r_k)}$ has arbitrarily small self interaction as $k$ gets large. By (\ref{bigradius}), $|B(x_k,r_k)| \leq 2/k$. Combining this with the fact that $K$ is bounded, we see that 
\begin{align*}
   |\mcE[\mathds{1}_{B(x_k,r_k)}]| \leq & ||K||_{\infty} |B(0,r_k)|^2 \\
   \leq & \frac{4||K||_{\infty}}{k^2}.
\end{align*}
Finally, we need to check that 
\begin{equation*}
    \lim_{k\to \infty} \mcE[\overline{\rho}_k^m]=\lim_{k\to \infty} \mcE[\rho_k^m].
\end{equation*}
We may write
\begin{equation}\label{eqn:expansion}
    \mcE[\overline{\rho}_k^m] = \mcE[\rho_k^m-(\rho_k^m-\overline{\rho}_k^m)]= \mcE[\rho_k^m] + \mcE[\rho_k^m-\overline{\rho}_k^m] -2\mcE[\overline{\rho}_k^m,\rho_k^m-\overline{\rho}_k^m]
\end{equation}
As before, using the fact that $K$ is bounded and $||\rho_k^m-\overline{\rho}_k^m||_{L^1(\mathbb{R}^d}\leq 1/k,$ it is straightforward to check that the second and third term in the right hand side of (\ref{eqn:expansion}) both tend to $0$ as $k\to \infty.$
\end{proof}

\section{Proof of Main Results}

Throughout this sections we will only be considering kernel which satisfy \textbf{(K1)-(K5)}.

\subsection{Separation Lemma}

The next Lemma will allow us to conclude that minimizers of the energy $\mcE$ cannot have points in their support within a certain range of ``forbidden" distances. This lemma therefore relates the minimization problem for general kernels to the toy problem in \cite{carrie-thesis}, which could be an avenue for sharpening the hypotheses on the size of the well and barrier required to get separation. 

\begin{lem}\label{lem:separation} Let $\rho \in \mcA_m$. Suppose $x_1,x_2\in \mathbb{R}^N$ satisfy
\begin{enumerate}
    \item $K*\rho(x_1),$ $K*\rho(x_2)\leq 0$, and \item $a+w \leq |x_1-x_2| \leq a+W-w$.
\end{enumerate}
Then, $x_1,x_2 \notin \supp(\rho)$.
\end{lem}

\begin{proof}
For $i=1,2$ we have 
\begin{align}
    0 \geq K*\rho(x_i)&= \int_{|x_i-y|\leq w}K(|x_i-y|)\rho(y)\, dy + \int_{|x_i-y|>w_!}K(|x_i-y|)\rho(y)\, dy  \nonumber \\ 
    & \geq -d \int_{|x_i-y|\leq w} \rho(y)\, dy + h \int_{a\leq |x_i-y|\leq a+W} \rho(y)\, dy. \nonumber 
\end{align}
Rearranging, we obtain
\begin{equation} \label{wellvsbarrier}
    \int_{a\leq |x_i-y|\leq a+W} \rho(y)\, dy \leq \frac{d}{h}\int_{|x_i-y|\leq w} \rho(y)\, dy
\end{equation}
This estimate says that $\rho$ cannot have too much mass in the annulus ...... relative to its mass near $x_i$.
Also note that for $i\neq j$
\begin{equation}\label{subsetinclusion}
    \int_{|x_i-y|\leq w}\rho(y)\, dy \leq \int_{a\leq |x_j-y|\leq a+W}\rho(y)\, dy
\end{equation}
since $a+w\leq |x_1-x_2|\leq a+W-w.$
Then, for $i\neq j$ we alternate between using (\ref{wellvsbarrier}) and (\ref{subsetinclusion}) to obtain
\begin{align*}
    \int_{|x_i-y|\leq w} \rho(y)\, dy & \leq \int_{a\leq |x_j-y|\leq a +W} \rho(y) \, dy \\
    &\leq \frac{d}{h}\int_{|x_j-y|\leq w} \rho(y) \, dy \\
    & \leq \frac{d}{h} \int_{a\leq |x_i-y|\leq a+W} \rho(y) \, dy \\
    & \leq \left(\frac{d}{h}\right)^2 \int_{|x_i-y|\leq w} \rho(y)\, dy.
\end{align*}
Then since $\rho\geq 0$ and $d<h,$ we must have 
\begin{equation*}
    \int_{|x_i-y|\leq w} \rho(y) \, dy =0.
\end{equation*}
\end{proof}

\subsection{Proof of Theorem for Compactly supported kernels}
In this section, we prove existence of minimizers for kernels which are $0$ outside of some large radius. Namely, we prove
\begin{thm}[Existence of minimizers for compactly supported kernels] \label{thm:exist}
Let $K:[0,\infty) \to \mathbb{R}$ be a bounded kernel satisfying $\mathbf{(K1)-(K3)}$. Suppose $K$ is zero outside of some large radius $R>0.$ Then for any $m>0$ the infimum 
\begin{equation*}
    \inf_{\rho \in \mcA_m} \mcE[\rho]
\end{equation*}
is attained by some $\rho \in \mcA_m .$ Moreover, if $K$ is strictly increasing on $[0,a]$ and satisfies $\mathbf{(K4),(K5)}$, then, up to sets of measure zero, any minimizer has the form
\begin{equation}\label{eqn:trunc min decomp}
    \rho=\sum_{i=1}^k \mathds{1}_{B(x_i,r_i)},
\end{equation}
for some centres $x_i\in \mathbb{R}^N$ and radii $r_i>0$ such that $\mcE[\mathds{1}_{B(x_i,r_i)},\mathds{1}_{B(x_j,r_j)}]=0$ for $i\neq j$.
\end{thm}

 The main tool we will use is Lions' concentration compactness principle (Lemma 1.1 in \cite{Lions_1984}), which will allow us to deal with the possibility that minimizing sequences can have pieces which get infinitely far away from one another.

\begin{lem}[Concentration Compactness] \label{lem:ConCom}
Let $\{\rho_n \}$ be a sequence in $L^1(\mathbb{R}^N)$ satisfying
\begin{equation*}
    \rho_n\geq0, \; \text{ and } \, \int_{\mathbb{R}^N}\rho_n(x)\, dx =m,
\end{equation*}
for a fixed $m>0.$ Then there exists a subsequence $\{\rho_{n_k}\}$ that satisfies one of the following three properties:
\begin{enumerate}
    \item (Tightness up to translation) There exists $y_k \in \mathbb{R}^N$ such that for every $\epsilon>0$ there exists an $R>0,$ 
    \begin{equation*}
        \int_{B(y_k,R}\rho_{n_k}(x) \, dx \geq m- \epsilon.
    \end{equation*}
    \item (Vanishing) For all $R>0$ \begin{equation*}
        \lim_{k\to \infty} \sup_{y\in \mathbb{R}^N} \int_{B(y,R)} \rho_{n_k}(x)\, dx =0.
    \end{equation*}
    \item  (Dichotomy) There exists an $0<\alpha<m,$ such that for any $\epsilon>0,$ there exist $k_0\geq 1,$ $y_k\in \mathbb{R}^N,$ and radii $R>0$ and $R_k \to \infty$ as $k\to \infty$ such that \begin{equation*}
        \rho_k^1= \rho_{n_k}|_B(y_k,R), \; \text{ and }\, 
        \rho_k^2= \rho_{n_k}|_{\mathbb{R}^N\setminus B(y_k,R_k)}
    \end{equation*}
    satisfy
    \begin{align*}
        & ||\rho_{n_k}-(\rho_k^1+\rho_k^2)||_{L^1(\mathbb{R}^N)} \leq \epsilon, \\
        & ||\rho_k^1||_{L^1(\mathbb{R}^N)}-\alpha\leq \epsilon, \; \text{ and} \\
        & ||\rho_k^2||_{L^1(\mathbb{R}^N)} -(m-\alpha) <\epsilon
    \end{align*}
    for $k\geq k_0.$
\end{enumerate}
\end{lem}

The statement for the dichotomy case here appears different from the statement in \cite{Lions_1984}, but the modification reflects the construction of $\rho_k^1,$ and $\rho_k^2$ in \cite{Lions_1984}.

The following Lemma will be used to rule out the vanishing case.
\begin{lem} \label{lem:nonvanish}
    Let the kernel $K$ be bounded and satisfy $\lim_{r\to \infty}K(r)=0.$ Suppose $\{\rho_k \}$ is a sequence in $\mcA_m$ that satisfies the vanishing property in the concentration compactness Lemma. Then, $\lim_{k\to \infty} \mcE [\rho]=0.$ 
\end{lem}
Note in particular that a minimizing sequence cannot be vanishing, nor can it have a ``vanishing part". 

\begin{proof}
Fix any $R>0.$ Then for any $\epsilon>0,$ using the fact that $\{ \rho_k\}$ is vanishing we have for large $k$ 
\begin{equation*}
    \int_{B(y,R)}\rho(x)dx<\epsilon
\end{equation*}
for any $y\in \mathbb{R}^N.$ We can now compute 
\begin{align*}
    |\mcE[\rho_k]| & \leq \left|\int_{\mathbb{R}^N}\int_{\mathbb{R}^N} K(|x-y|)\mathds{1}_{|x-y|\leq R} \rho_k(x)\rho_k(y)\; dx \; dy  \right| \\ & \;\;\;\;\;\;\;\;+ \left|\int_{\mathbb{R}^N}\int_{\mathbb{R}^N} K(|x-y|)\mathds{1}_{|x-y|> R} \rho_k(x)\rho_k(y)\; dx \; dy  \right| \\
    & \leq ||K||_{L^{\infty}(\mathbb{R}^N)} \int_{\mathbb{R}^N} \rho_k(y) \int_{B(y,R)} \rho_k(x) \; dx \; dy + m^2 \sup_{r\geq R} K(r)  \\
    & \leq ||K||_{L^{\infty}(\mathbb{R}^N)} m \epsilon + m^2 \sup_{r\geq R} K(r)
\end{align*}
The second term of this final sum can be made arbitrarily small since $\lim_{r\to \infty} K(r)=0.$\end{proof}

This Lemma is used to extract droplets:

\begin{lem}\label{lem:tight->min}
Suppose $\{\rho_k \}$ is a sequence of densities in $\mcA_m$. If $\{\rho_k\}$ is tight up to translation, then there is a $\rho \in \mcA_m$ such that
\begin{equation*}
    \lim_{k\to \infty} \mcE[\rho_k]= \mcE[\rho].
\end{equation*}
\end{lem}

\begin{proof}
Follows similar argument as in \cite{Choksi_Fetecau_Topaloglu_2015}.
By tightness up to translation, there exist $y_k\in \mathbb{R}^N$ such that for any $\epsilon>0,$ there is a radius $R>0$ such that 
\begin{equation}\label{eq:tight}
    \int_{B(y_k,R)}\rho_k(y)\, dy \geq m-\epsilon
\end{equation}
for all $k.$ By the translation invariance of the energy, without loss of generality we can take $y_k=0$ for all $k.$ Fix any $1<q<\infty,$ then $\{\rho_k\}$ is a bounded sequence in $L^q(\mathbb{R}^N),$ so there is a $\rho \in L^q(\mathbb{R}^N)$ such that $\rho_k$ converges to  $\rho$ weakly in $L^q.$ 

First, we will verify that $\rho \in \mcA_m$. By (\ref{eq:tight}), 
\begin{equation*}
    \int_{\mathbb{R}^N} \rho(y) \, dy =m.
\end{equation*}
To see that $\rho \geq 0,$ consider the set $A=\{ \rho < 0 \}$. Suppose $|A|>0$, and if $|A|=\infty,$ replace $A$ with a subset that has finite measure. Then, by weak convergence we have
\begin{equation*}
    0\leq \lim_{k\to \infty} \int_A \rho_k(y)\, dy = \int_A \rho(y)\, dy <0,
\end{equation*}
which is a contradiction. Therefore $|A|=0$. Similarly, if we let $B=\{ \rho >1 \},$ and assume $|B|>0$, we can compute
\begin{equation*}
     |B| \geq \lim_{k\to \infty} \int_B \rho_k(y) \, dy = \int_B \rho(y) \, dy > |B|.
\end{equation*}
Combined, this means that $\rho \in \mcA_m$

Next, we will prove that $\lim_{k\to \infty} \mcE[\rho_k] = \mcE[\rho].$
Let
\begin{align*}
    & G_k(x)=K*\rho_k (x),\, \text{ and} \\
    & G(x)=K*\rho(x).
\end{align*}
By (\ref{eq:tight}), $\rho_k$ also converges to $\rho$ weakly in $L^1(\mathbb{R}^N).$ Then since $K$ is bounded, this means that $G_k$ converges to $G$ pointwise.

Fix $\epsilon >0,$ and pick $R>0$ such that 
\begin{equation} \label{eqn:bigradius}
    \int_{\mathbb{R}^N\setminus B(0,R)}\rho_k(y) \, dy \leq \epsilon
\end{equation}
for all $k.$ Then compute
\begin{align}
    \mcE[\rho_k]-\mcE[\rho] = &  \int_{\mathbb{R}^N}G_k(x)\rho_k(x)\, dx - \int_{\mathbb{R}^N}G(x) \rho(x) \, dx  \nonumber \\
    = &  \int_{B(0,R)} (G_k(x)-G(x))\rho_k(x) \, dx +\int_{B(0,R)} G(x)(\rho_k(x)-\rho(x)) \, dx  \label{eqn:sum} \\
    & +  \int_{\mathbb{R}^N\setminus B(0,R)} G_k(x) \rho_k(x)\, dx. -\int_{\mathbb{R}^N\setminus B(0,R)} G(x)\rho(x) \, dx . \nonumber
\end{align}
The first term in (\ref{eqn:sum}) converges to $0$ as $k\to \infty$ by the bounded convergence theorem. The second term also converges to $0$ as $k\to \infty$ since $G|_{B(0,R)}$ is an admissible test function, noting that $||G||_{L^{\infty}(\mathbb{R}^N)}\leq  m ||K||_{L^{\infty}(\mathbb{R}^N)}$ 
Thus, for large enough $k$ we have
\[
\left| \mcE[\rho_k]-\mcE[\rho] \right| \leq 2\epsilon + 2m||K||_{L^{\infty}(\mathbb{R}^N)} \epsilon.
\]
\end{proof}

The next Lemma will allow us to separate a minimizing sequence into multiple pieces, in the case where we get a dichotomous subsequence after applying the concentration compactness principle.
\begin{lem}\label{lem:dichotomy->two pieces}
Suppose $\{\rho_k\}$ is a dichotomous sequence in $\mcA_m.$ Then, for some $0<\alpha<m$ there are sequences $\{\rho_k^1 \}$ and $\{\rho_k^2\}$ in $\mcA_{\alpha}$ and $\mcA_{m-\alpha},$ respectively, such that
\begin{equation}
    \lim_{k\to \infty} \left(\mcE[\rho_k] -\mcE[\rho_k^1] -\mcE[\rho_k^2] \right)=0.
\end{equation}
Moreover if $\rho_k$ is a minimizing sequence, then
\begin{equation}
    \liminf_{k\to \infty} \mcE[\rho_k^1]=\inf_{\rho\in\mcA_{\alpha}}\mcE[\rho],
\end{equation}
and 
\begin{equation}
    \liminf_{k\to \infty} \mcE[\rho_k^2]=\inf_{\rho\in\mcA_{m-\alpha}}\mcE[\rho].
\end{equation}
\end{lem}

\begin{proof}
 By dichotomy, after perhaps passing to a subsequence, we can find radii $R_k>0$ and points $y_k\in \mathbb{R}^N$ so that the sequences
\begin{equation}
    \rho_k^1=\rho_k|_{B(y_k,R_k)}, \; \text{ and } \; \rho_k^2=\rho_k|_{\mathcal{R}^N \setminus B(y_k,R_k+\Tilde{R})}
\end{equation}
satisfy
\begin{align}
    & \lim_{k\to \infty}\int \rho_k^1 =\alpha, \; \text{ and} \\
    & \lim_{k\to \infty}\int \rho_k^2 =m-\alpha
\end{align}
for some $0<\alpha<m.$
Note that since $K(r)=0$ for $r\geq R,$
\begin{equation*}
    \mcE[\rho_k^1,\rho_k^2]=0.
\end{equation*}
So,
\begin{equation}\label{eqn:decomp}
    \mcE[\rho_k]=\mcE[\rho_k^1]+\mcE[\rho_k^2]+\mcE[\rho-\rho_k^1-\rho_k^2]+2\mcE[\rho_k^1+\rho_k^2, \rho-\rho_k^1-\rho_k^2] 
\end{equation}
Note that as $k\to \infty$ the last two terms approach $0$ since 
\begin{equation*}
    \lim_{k\to \infty} \int_{\mathbb{R}^N}\rho-\rho_k^1-\rho_k^2 = m-\alpha-(m-\alpha)=0.
\end{equation*}
Next, we will apply concentration compactness to $\{\rho_k^1 \}$ and $\{\rho_k^2\}.$ In order to do this, we must first modify these sequences so that they have constant mass. To do so for $\rho_k^1$, we either decrease the radius $R_k$ if $\rho_k^1$ has too much mass, or add a small piece far away if $\rho_k^1$ has too little mass. For each $k,$ let 
\begin{equation*}
    \overbar{\rho_k^1}=\rho_k|_{B(y_k,\overbar{R_k^1})}+\mathds{1}_{B(x_k,r_k^1)},
\end{equation*} where $\overbar{R_k^1}\leq R_k,$ $r_k^1\geq 0,$ and $x_k\in \mathbb{R}^N$ are chosen so that $\overbar{\rho_k^1}\in \mcA_{\alpha},$
and $|x_k-y_k|>R_k+r_k^1+\Tilde{R}$, and $r_k^1 \to 0$ as $k\to \infty.$ 
Similarly for $\rho_k^2,$ we either increase the radius $R_k$ to decrease the mass, or add a small ball centered at $y_k$ to increase the mass. That is we let
\begin{equation*}
    \overbar{\rho_k^2}=\rho_k|_{\mathbb{R}^N \setminus B(y_k \overbar{R_k^2}+\Tilde{R})}+\mathds{1}_{B(y_k,r_k^2)},
\end{equation*}
where $\overbar{R_k^2}\geq R_k$, and $r_k^2\geq 0$ are chosen so that $\overbar{\rho_k^2}\in \mcA_{m-\alpha},$
and $r_k^2 \to 0$ as $k\to \infty.$
It is straightforward to check, using the fact that $||\rho_k^i-\overbar{\rho_k^i}||_{L^1(\mathbb{R}^N)} \to 0$ as $k\to \infty$ for $i=1,2$ that
\begin{equation*}
    \lim_{k\to \infty} \left(\mcE[\rho_k^i]-\mcE[\overbar{\rho_k^i}] \right)=0.
\end{equation*}
Then, combining this with (\ref{eqn:decomp}), 
\begin{equation} \label{eqn:decomplimit}
    \lim_{k\to \infty} \left(\mcE[\rho_k]-\mcE[\overbar{\rho_k^1}]-\mcE[\overbar{\rho_k^2}] \right)=0.
\end{equation}

Now, suppose $\rho_k$ is a minimizing sequence. By subadditivity (Lemma \ref{lem:subadd}), we have
\begin{equation*}
    \lim_{k\to \infty} \mcE[\rho_k]=\inf_{\rho\in\mcA_m} \mcE[\rho] \leq \inf_{\rho\in\mcA_{\alpha}} \mcE[\rho] + \inf_{\rho\in\mcA_{m-\alpha}} \mcE[\rho].
\end{equation*}
By ($\ref{eqn:decomplimit}$)
\begin{align*}
    \lim_{k\to \infty} \mcE[\rho_k] = & \lim_{k\to\infty} \left( \mcE[\overbar{\rho_k^1}]+\mcE[\overbar{\rho_k^2}] \right) \\
    \geq & \liminf_{k\to\infty}\mcE[\overbar{\rho_k^1}] +\liminf_{k\to\infty}\mcE[\overbar{\rho_k^2}].
\end{align*}
Then we are done since 
\begin{equation*}
    \liminf_{k\to\infty}\mcE[\overbar{\rho_k^1}] \geq \inf_{\rho\in\mcA_{\alpha}} \mcE[\rho],
    \; \text{ and } \, \liminf_{k\to\infty}\mcE[\overbar{\rho_k^1}] \geq \inf_{\rho\in\mcA_{m-\alpha}} \mcE[\rho].
\end{equation*}
\end{proof}

\begin{proof}[Proof of Theorem \ref{thm:exist}] Let $\left\{ \rho_k \right\}$ be a minimizing sequence. 

Then by concentration compactness (Lemma \ref{lem:ConCom}), we can pass to a subsequence to a obtain a minimizing sequence which, abusing notation slightly, I will also label $\{ \rho_k \}$ that is either tight up to translation, vanishing, or dichotomous. By Lemma [\ref{lem:nonvanish}] a minimizing sequence cannot be vanishing since we know $\inf_{\rho\in\mcA_m}\mcE[\rho]<0.$

If the sequence is tight up to translation, we can apply Lemma $\ref{lem:tight->min},$ and conclude that there is a $\rho\in\mcA_m$ such that $\lim_{k \to \infty}\mcE[\rho_k]=\mcE[\rho],$ and so $\rho$ is a minimizer and we are done. 

If the sequence is dichotomous, then by Lemma \ref{lem:dichotomy->two pieces} there are $\rho_k\in \mcA_{\alpha_1},$ and $\rho_k\in\mcA_{\alpha_2},$ with $\alpha_1+\alpha_2=m,$ and
\begin{equation*}
    \lim_{k\to \infty} \mcE[\rho_k] = \lim_{k\to \infty}\left( \mcE[\rho_k^1]+\mcE[\rho_k^2] \right).
\end{equation*}
We then apply concentration compactness principle to $\{ \rho_k^1 \}$ and $\{ \rho_k^2\},$ noting that neither can have a vanishing subsequence. Applying Lemmas \ref{lem:tight->min} and \ref{lem:dichotomy->two pieces}, and then concentration again iteratively, we eventually obtain (after passing to subsequences, and some relabelling) sequences $\rho_k^1,\rho_k^2,\dots \rho_k^l,$ which are tight up to translation with $\rho_k^i\in \mcA_{\alpha_i},$ $\alpha_1+\dots+\alpha_l=m,$ and
\begin{equation*}
    \lim_{k\to \infty} \mcE[\rho_k] = \lim_{k\to \infty}\left( \mcE[\rho_k^1]+\dots+\mcE[\rho_k^l] \right).
\end{equation*}
Note that this process of applying concentration compactness iteratively must end after finitely many steps, because the $\alpha_i$'s cannot become arbitrarily small. This is because by Lemma \ref{lem:smallmassballmin} $\inf_{\rho\in\mcA_{\alpha}}\mcE[\rho]$ is attained by the indicator of a ball if $\alpha \leq m_0.$ So for any $\alpha \leq m_0$, and any dichotomous sequence $\{\eta_k\}$ in $\mcA_{\alpha},$ we would have
\begin{equation*}
    \liminf_{k\to \infty} \mcE[\eta_k]>\inf_{\rho \in \mcA_{\alpha}}\mcE[\rho],
\end{equation*}
which contradicts Lemma \ref{lem:dichotomy->two pieces}.
To conclude, by Lemma \ref{lem:tight->min} there exist $\rho^1\in\mcA_{\alpha_1},\dots\rho^l\in \mcA_{\alpha_l}$ such that 
\begin{equation*}
    \lim_{k\to \infty} \mcE[\rho_k] = \mcE[\rho^1] + \dots +\mcE[\rho^l].
\end{equation*}
Each $\rho^i$ is a minimizer of $\mcE$ in $\mcA_{\alpha_i}$, and so by Lemma \ref{lem:mins have compact support}, they each have compact support. So, we can construct a minimizer for the truncated problem by letting $\rho=\rho(z_1+\cdot)+ \dots \rho^l(z_l+\cdot)$ for suitably chosen $z_1,\dots ,z_l\in\mathbb{R}^N$.

Finally, assume additionally that $K$ is strictly increasing on $[0,a],$ and $a+w\leq W-2w$. Let $\rho$ be any minimizing density. By Lemma $\ref{lem:separation},$ any two points $x,y$ points in the support of $\rho$ satisfy the distance condition $|x-y|\notin [a+w,a+W-w].$ Then since $a+w\leq W-2w=(a+W-w)-(a+w),$ $\supp{\rho}$ can be decomposed into pieces each with diameter at most $a+w$. On each of these pieces, the energy is uniquely minimized (up to translations and modifications on sets of measure zero) by the indicator of a ball. 
\end{proof}

\subsection{Proof of Theorem \ref{thm:inf is many balls}}\label{section:gen min}

The following notion of a generalized minimizer comes from \cite{Knupfer_Muratov_Novaga_2016}. It allows us to deal with the fact that minimizing sequences may have many pieces that become infinitely far apart.
\begin{defn}
A \textbf{generalized minimizer} of $\mcE$ in $\mcA_m$ is a collection of densities  $(\rho_1,\dots,\rho_M),$ for some $M\in \mathbb{N},$ and each $\rho_i \in \mcA_{m_i}$ is a minimizer for $\mcE$ in $\mcA_{m_i}.$ Additionally, $\sum_{i=1}^M m_i=m,$ and 
\begin{equation}
    \inf_{\rho\in \mcA_m}\mcE[\rho]=\sum_{i=1}^M \mcE[\rho_i].
\end{equation}
\end{defn}
Let $\overbar{K}=K\cdot \mathds{1}_{[0,a+W]},$ and denote 
\begin{equation}
    \overbar{\mcE}[\rho]=\int_{\mathbb{R}^N} \int_{\mathbb{R}^N} \overbar{K}(|x-y|)\rho(x)\rho(y) \, dx \, dy.
\end{equation}

\begin{thm}[Existence of Generalized Minimizers]\label{thm:exist gen} Let $K:[0,\infty) \to \mathbb{R}$ be a bounded kernel satisfying $\mathbf{(K1)-(K5)}$. 
Further, assume $K$ is strictly increasing on $[0,a]$ and $a+w\leq W-2w$. then, generalized minimizers exist and up to sets of measure zero any generalized minimizer has the form \begin{equation}
    \rho=(\mathds{1}_{B(x_i,r_i)})_{i=1}^k
\end{equation}
for some centres $x_i\in \mathbb{R}^N$ and radii $r_i>0$ such that $\mcE[\mathds{1}_{B(x_i,r_i)},\mathds{1}_{B(x_j,r_j)}]=0$ for $i\neq j$.

\end{thm}

\begin{proof}
    Let $\rho\in \mcA_m$ be a minimizer for the truncated problem. Note that $\overbar{K} \leq K,$ so $\overbar{\mcE}[\eta]\leq \mcE[\eta],$ for any $\eta \in \mcA_m.$ 
    
    By (\ref{eqn:trunc min decomp}), we may write \begin{equation}
        \rho=\sum_{i=1}^k \mathds{1}_{B(x_i,r_i)},
    \end{equation}
    for some centres $x_i\in \mathbb{R}^N$ and radii $r_i>0$ such that $\mcE[\mathds{1}_{B(x_i,r_i)},\mathds{1}_{B(x_j,r_j)}]=0$ for $i\neq j$.
    From this, we may construct a minimizing sequence for $\mcE$, by picking new centers $x_{i,l}$ so that $|x_{i,l}-x_{j,l}|\to \infty$ as $k\to \infty$. So,
    \begin{equation*}
        \inf_{\rho\in \mcA_m}\mcE[\rho]=\inf_{\rho \in \mcA_m} \overbar{\mcE}[\rho].
    \end{equation*}
    So, by taking $\rho_i=\mathds{1}_{B(x_i,r_i)}$, we obtain a generalized minimizer for $\rho$
\end{proof}

\subsection{Linear Growth}\label{subsec:lingrowth}

We will show that the minimal energy
\begin{equation*}
    E(m)=\inf_{\rho \in \mcA_m} \mcE[\rho]
\end{equation*}
grows linearly in $m$.
\begin{thm} Let $K:[0,\infty) \to \mathbb{R}$ be a bounded kernel satisfying $\mathbf{(K1)-(K5)}$. 
Further, assume $K$ is strictly increasing on $[0,a]$ and $a+w\leq W-2w$. For each $m>0,$ let 
$E(m)=\inf_{\rho \in \mcA_m} \mcE[\rho]$
and $g(m)$ be the energy of the ball with mass $m$. Then
\begin{equation}  \label{eqn:linear growth}
    \lim_{m\to \infty} \frac{E(m)}{m} = \min_{m>0}\frac{g(m)}{m}.
\end{equation}
\end{thm}

\begin{proof}
The first step is to show that the limit in the left hand side of (\ref{eqn:linear growth}) exists. By Fekete's subadditive Lemma (see \cite{Kuczma} Theorem 16.2.9) we have
\begin{equation*}
    \lim_{m\to \infty} \frac{E(m)}{m} = \inf_{m>0} \frac{E(m)}{m}.
\end{equation*}
Next, by Theorem \ref{thm:inf is many balls} given any $m>0$ we may write $E(m)=g(m_1)+g(m_2)+ \dots +g(m_k)$ for some positive numbers $m_1,\dots m_k$ such that $m_1+\dots+m_k=m.$ Then,
\begin{equation*}
    \frac{E(m)}{m}=\frac{g(m_1)+\dots + g(m_k)}{m} \geq \frac{(m_1+ \dots +m_k)\inf_{m>0}\frac{g(m)}{m}}{m}=\inf_{m>0}\frac{g(m)}{m}
\end{equation*}
On the other hand, $E(m)\leq g(m)$ for all $m$, so 
\begin{equation*}
    \inf_{m>0}\frac{E(m)}{m} =\inf_{m>0} \frac{g(m)}{m}.
\end{equation*}
To conclude we need to show that the infimum on the right is actually a minimum. Note that $g$ is continuous on $(0,\infty),$ and is positive for large values of $m$. Also, since $K(r)\geq K(0)$ for all $r>0,$
\begin{equation*}
    g(m) \geq K(0)m^2.
\end{equation*}
Then, using the squeeze theorem, and the fact that $g$ is negative for small values of $m,$
\begin{equation*}
    \lim_{m \to 0^+} \frac{g(m)}{m}=0.
\end{equation*}

\end{proof}

\section{Size of droplets}\label{sec:sizeofdrop}

Theorem \ref{thm:exist gen} says that generalized minimizers are tuples of indicators of balls. The linear growth estimate (\ref{eqn:linear growth}) suggests that when $m$ is larger enough, the droplets should all be relatively close in size. In this section we will consider kernels that are power laws near $0,$ to further explore the question of droplet size. We will see that generalized minimizers for such kernels consist of many indicators of balls of one size, and possibly one ball of a smaller size. 

Let
\begin{equation}
    K(r)=
        \left\{
        \begin{array}{lr}
        r^p-d, & \; \text{ if } 0\leq r \leq a,\\
        f(r), & \; \text{ if } r>a,
        \end{array} 
        \right.
\end{equation}
where $f:(a,\infty) \to \mathbb{R}$ is a non-negative function such that $f(r)\geq a^p+d$ for all $a<r\leq w+a$ and such that $\lim_{r\to \infty}f(r)=0$. To understand generalized minimizers for such a kernel, we first compare the indicator of one ball versus two balls.

\begin{prop}\label{lem:1 ball vs 2}
Let $g(m)=\mcE[\mathds{1}_{B_{r(m)}}]$, where $r(m)$ is the radius of the ball with measure $m$, and set $f(t)=g(tm)+g((1-t)m)$ for $0\leq t \leq 1/2$. Then, there are numbers $0<m_0<m_1$ such that
\begin{enumerate}
    \item If $0<m\leq m_0,$ then $f$ is minimized at $t=0$
    \item If $m_0<m<m_1$ $f$ is minimized at some $0<t_0<1/2$
    \item If $m\geq m_1$, then $f$ is minimized at $t=1/2.$
\end{enumerate}
Moreover, 
\begin{align}
    m_0=&\left(\frac{2d}{C_{n,p}(2+p/n)}  \right)^{n/p}, \text{ and} \\
    m_1=&\left(\frac{2^{1+p/n}d}{C_{n,p}(2+p/n)(1+p/n)} \right)^{n/p} ,
\end{align}
where
\begin{equation*}
    C_{n,p}=\frac{1}{|B_1|^{2+p/n}}\int_{B_1}\int_{B_1} |x-y|^p \, dx \, dy.
\end{equation*}
\end{prop}

\begin{proof}
Let $m>0$ and $r>0$ such that $m=|B_r|.$ Then,
\begin{align*}
    g(m)= \mcE[\mathds{1}_{B_r}] & = \int_{B_r}\int_{B_r} \left( |x-y|^p-d\right)\, dx \, dy \\
    &= r^{2n+p}\int_{B_1}\int_{B_1} |x-y|^p \, dx \, dy \, -dm^2 \\
    &= C_{n,p} m^{2+p/n}-dm^2.
\end{align*}
We seek to minimize $f(t)=g(tm)+g((1-t)m)$ for $t\in [0,1/2].$ Compute
\begin{align*}
    f'(t)= & m^2 \left(C_{n,p} (2+p/n)m^{p/n}(t^{1+p/n}-(1-t)^{1+p/n}) -2d(2t-1) \right) \\
    f''(t) = & m^2 \left( C_{n,p}(2+p/n)(1+p/n) m^{p/n}(t^{p/n}+(1-t)^{p/n})-4d \right)
\end{align*}
It is straightforward to check that $t\mapsto t^{p/n}+(1-t)^{p/n}$ is decreasing on $[0,1/2]$ if $p/n>1$. So, $f''$ is decreasing, which means $f'$ is concave. Note that $f'(1/2)=0,$ so $f$ can have at most one critical value in the interval (0,1/2). To determine whether such a critical value exists, we need to look at the signs of $f'(0)$ and $f''(1/2)$. 
\begin{align*}
    f'(0)=& m^2C_{n,p}(2+p/n)(m_0^{p/n}-m^{p/n}), \\
    f''(1/2)=& 2^{1-p/n}m^2C_{n,p}(2+p/n)(1+p/n)(m^{p/n}-m_1^{p/n}).
\end{align*}
Let $m_0$ and $m_1$ be as defined in Lemma $\ref{lem:1 ball vs 2}$. Then,
\begin{enumerate}
    \item If $0<m\leq m_0,$ then $f'(0)\geq 0$, and so $f$ has no critical values in $(0,1/2).$ Moreover, $f'(t)\geq0$ for all $t\in [0,1/2].$ So, $f$ has its minimum at $t=0.$
    \item If $m_0<m<m_1,$ then $f'(0)<0$ and $f''(1/2)<0.$ Then there is a point $0<t_0<1/2$ where $f(t_0)=0.$ Moreover, $f'(t)<0$ on $(0,t_0)$ and $f'(t)>0$ on $(t_0,1/2)$. So, $f$ has its minimum at $t=t_0.$
    \item If $m\geq m_1,$ then $f''(1/2)\geq 0$, which means $f'(t)\leq 0$ for all $t\in [0,1/2].$ Thus $f$ has its minimum at $t=1/2$
\end{enumerate}
Note, it is straightforward to check that $m_0<m_1$

\end{proof}

\begin{lem}
Let $m>0$ and $(\rho_i)_{i=1}^k$ be a generalized minimizer. Then there exist radii $r_1,r_2> 0$ such that each $\rho_i=\mathds{1}_{B_{r_1}}$ or $\mathds{1}_{B_{r_2}}$.
\end{lem}

\begin{proof}
Let $m>0$ $(\rho_i)_{i=1}^k$ be a generalized minimizer. Consider the minimization problem
\begin{equation}\label{eqn:min k balls}
    \min_{\substack{t_1+\dots +t_{k-1}< 1  \\
     t_i > 0}} g(t_1m)+ \dots g(t_{k-1}m) + g((1-t_1-\dots-t_{k-1})m).
\end{equation}
Then, $(\rho_i)_{i=1}^k$ corresponds to a minimizer $(t_1, \dots, t_{k-1})$ of (\ref{eqn:min k balls}). Computing the gradient yields
\begin{equation*}
    mg'(t_im)-mg'((1-t_1-\dots-t_
    {k-1})m)=0
\end{equation*}
for each $i=1,\dots,{k-1}.$ In particular this means that
\begin{equation*}
    g'(t_im)=g'(t_jm)
\end{equation*}
for each $i\neq j.$
Note that $g'$ strictly decreasing then increasing on $[0,\infty),$ so for any $y\in \mathbb{R}$ the equation $g'(x)=y$ has at most two solutions. 

\end{proof}

Note that the only thing we used was that $g$ is concave then convex so this argument applies more generally than to just the power law kernels.

\begin{thm}\label{lem:power law one small ball}
Generalized minimizers have the form $(\mathds{1}_{B_s},\mathds{1}_{B_r},\dots\mathds{1}_{B_r})$, for some $s \leq r.$ 
\end{thm}
\begin{proof}
In the previous Lemma we saw that generalized minimizers will consist of indicators of balls having at most two different sizes. 

Let $m_0$ and $m_1$ be the thresholds from Lemma \ref{lem:1 ball vs 2}. 

The claim follows by induction on the number of balls in a generalized minimizer. First, for a given $m>0$
assume that a generalized minimizer consists of four balls, two of which have mass $M,$ and the other two have size $N$. Then, by Lemma \ref{lem:1 ball vs 2} $2M\geq m_1,$ $2N\geq m_1,$ and $M+N\leq m_1$. So,
\begin{equation*}
    2m_1\leq 2M+2N \leq 2m_1,
\end{equation*}
 which means $M=N.$ 
 
 For the inductive step, assume the claim is true for generalized minimizers that consist of $k$ balls. Then, let $(\rho_i)_{i=1}^{k+1}$ be a generalized minimizer that consists of $k+1$ indicators balls, $k_M>1$ of mass $M$ and $k_N$ of mass $N.$ Then, removing on of the balls of mass $M$ will give a generalized minimizer for the problem with with smaller mass, and so $k_N=1.$

 To see that the single ball must have smaller mass, consider a generalized minimizer $(\mathds{1}_{B_s},\mathds{1}_{B_r},\dots\mathds{1}_{B_r}),$ for some $s\neq r$. By Lemma \ref{lem:1 ball vs 2}, $|B_s|+|B_r|<m_1$ and $2|B_r|\geq m_1,$ so $|B_s|<|B_r|.$
\end{proof}

\bibliographystyle{amsalpha}
\bibliography{ref}

\end{document}